\documentclass[11pt]{amsart}
\usepackage{amsmath,amsthm,amsfonts,latexsym,amssymb,amscd,color,enumerate}
\usepackage{tikz}
\usetikzlibrary{patterns}
\usepackage{chemarr}
\newtheorem{thm}{Theorem}[section] 
\newtheorem*{thm*}{Theorem}
\newtheorem{cor}[thm]{Corollary}
\newtheorem{lm}[thm]{Lemma}
\newtheorem*{lm*}{Lemma}

\newtheorem{clm}[thm]{Claim}

\newtheorem*{clm*}{Claim}
\theoremstyle{definition}
\newtheorem{df}[thm]{Definition}

\newtheorem{quest}[thm]{Question}
\numberwithin{equation}{section}
\newcommand{\sprf}{\noindent{\it Proof.}} 
\newcommand{\sqed}{\hfill\rule{1.3mm}{3mm}\medskip}

\newcommand{\cproof}{\noindent{\it Proof of Claim.}\ } 

\newcommand{\cqed}{\hfill\rule{1.3mm}{3mm}}
\newenvironment{cpf}{\cproof }{\cqed}
\newcommand{\bd}{\begin{description}}
\newcommand{\ed}{\end{description}}
\newcommand{\cf}{{\rm cf}}

\makeatletter
\newcommand*{\defeq}{\mathrel{\rlap{%
                     \raisebox{0.3ex}{$\m@th\cdot$}}%
                     \raisebox{-0.3ex}{$\m@th\cdot$}}%
                     =}
\makeatother

\begin{document}

\title{J\'{o}nsson J\'{o}nsson-Tarski Algebras}

\author{Jordan DuBeau}
\address[Jordan DuBeau]{Department of Mathematics\\
University of Colorado\\
Boulder, CO 80309-0395\\
USA}
\email{jordan.dubeau@colorado.edu}


\begin{abstract}
  By studying the variety 
  of J\'{o}nsson-Tarski algebras,
  we demonstrate two obstacles to the existence of large 
  J\'{o}nsson algebras in certain varieties. First, 
  if an algebra $J$ in a language $L$
  has cardinality greater than $|L|^+$ and a distributive 
  subalgebra lattice,
  then it must have a proper subalgebra of size $|J|$. 
  Second, if an algebra $J$ in a language $L$ 
  satisfies $\text{cf}(|J|) > 2^{|L|^+}$ and lies in a 
  residually small variety, then it again must have a 
  proper subalgebra of size $|J|$.
  We apply the first 
  result to show that 
  J\'{o}nsson algebras in the variety of J\'{o}nsson-Tarski algebras 
  cannot have 
  cardinality greater than $\aleph_1$. 
  We also
  construct $2^{\aleph_1}$ many pairwise nonisomorphic 
  J\'{o}nsson algebras in this variety, thus proving that for 
  some varieties the maximum possible number 
  of J\'{o}nsson algebras can be achieved.
\end{abstract}

\maketitle

\section{Introduction.}

A \textit{J\'{o}nsson  algebra} is an infinite algebra $J$, in a countable
algebraic language, which has no proper subalgebras of the same 
cardinality as $J$. While there exist some trivial examples of 
countable J\'{o}nsson algebras, 
uncountable J\'{o}nsson algebras are more difficult to construct. 
Some results about J\'{o}nsson algebras are surveyed in \cite{coleman}.

A \textit{J\'{o}nsson-Tarski algebra} is an algebra $A$ with one binary 
operation $\cdot$ and two unary 
operations $\ell$ and $r$, satisfying the identities 
\begin{enumerate}
\item $\ell(x \cdot y) = x$,
\item $r(x \cdot y) = y$, and 
\item $\ell(z) \cdot r(z) = z$.
\end{enumerate}
The identities express that $\cdot$ is a bijection 
$A \times A \to A$, with inverse map $z \mapsto (\ell(z), r(z))$.
J\'{o}nsson-Tarski algebras were first introduced in \cite{jonsson-tarski}.

The two concepts above were connected in \cite{dubeau-kearnes},
where together with K. Kearnes we asked the
question: can there exist an uncountable J\'{o}nsson algebra in a 
minimal variety? As 
it turns out, the variety of J\'{o}nsson-Tarski algebras is a minimal variety,
and so by constructing a J\'{o}nsson J\'{o}nsson-Tarski algebra of cardinality 
$\aleph_1$, we answered our question positively. Moreover we saw
that the J\'{o}nsson property interacted with other algebraic 
properties in nontrivial ways: for example, while a J\'{o}nsson algebra 
can exist in a minimal variety, a residually finite J\'{o}nsson algebra 
cannot exist in a minimal variety. And if residually finite J\'{o}nsson 
algebras exist at all, then their cardinality will be bounded 
by $2^{2^\omega}$.

After constructing a J\'{o}nsson J\'{o}nsson-Tarski algebra of cardinality $\aleph_1$,
we naturally wondered if the construction could be extended 
to larger cardinalities. We were inspired, for example, by 
P. Erd\H{o}s and A. Hajnal's work in \cite{erdos-hajnal}
in which the authors moved inductively from a J\'{o}nsson algebra of 
cardinality $\aleph_k$ to a J\'{o}nsson algebra of cardinality 
$\aleph_{k+1}$, thus generating a J\'{o}nsson algebra of every 
cardinality $\aleph_n$, $n \in \omega$. There is also the paper 
of S. Shelah \cite{shelah1} in which the author constructs a 
J\'{o}nsson group of cardinality $\aleph_1$ and remarks that a 
similar construction produces a J\'{o}nsson group of cardinality 
$\aleph_2$. Yet our efforts to construct a J\'{o}nsson J\'{o}nsson-Tarski algebra of 
any cardinality larger than $\aleph_1$ failed, leading us to 
ask what exactly the obstacles might be to the construction 
of J\'{o}nsson algebras in the particular 
variety of J\'{o}nsson-Tarski algebras, and what more can be learned 
about J\'{o}nsson algebras from 
this variety.

In Section \ref{subdist} we prove that J\'{o}nsson J\'{o}nsson-Tarski 
algebras 
cannot have cardinality greater than $\aleph_1$. This comes as a 
corollary to a more general result, that for an algebra 
$J$ of size $\kappa$ in a language of size $\lambda$, 
if $\kappa > \lambda^+$ and the subalgebra lattice of $J$ 
is distributive, then $J$ must have a proper subalgebra 
of cardinality $\kappa$. J\'{o}nsson-Tarski algebras always have 
distributive subalgebra lattices, and their language is 
countable, so the maximum possible cardinality of a 
J\'{o}nsson algebra in this variety is $\aleph_1$. This answers 
the question of exactly which cardinalities of J\'{o}nsson J\'{o}nsson-Tarski 
algebras are possible and reveals an obstacle to the existence of 
large J\'{o}nsson algebras in this variety and others.

In Section \ref{resid} we prove: for an algebra $J$ of size
$\kappa$ in a language of size $\lambda$, if $\cf(\kappa) 
> 2^{\lambda^+}$ and $J$ 
lies in a residually small 
variety, then $J$ must have a proper subalgebra of 
cardinality $\kappa$. This is a generalization of our 
work in 
\cite{dubeau-kearnes} concerning residually finite J\'{o}nsson 
algebras.
The new result now applies to J\'{o}nsson 
J\'{o}nsson-Tarski algebras, since the variety of J\'{o}nsson-Tarski algebras is 
residually small. While it does not give as tight of a 
bound on the cardinality of J\'{o}nsson J\'{o}nsson-Tarski algebras as the result 
of Section \ref{subdist} does, it has the benefit of applying 
to a much larger class of varieties. In other words, it shows 
a much more common 
obstacle to the existence of large J\'{o}nsson algebras. 

Finally in Section \ref{JJT} we give a system of constructing 
J\'{o}nsson J\'{o}nsson-Tarski algebras of cardinality $\aleph_1$ which is based 
on the construction found in \cite{dubeau-kearnes}.
This system of construction yields 
$2^{\aleph_1}$ many pairwise
nonisomorphic J\'{o}nsson J\'{o}nsson-Tarski algebras. 
We consider the question of how many 
J\'{o}nsson algebras exist within a variety -- not just whether they 
exist -- to be an 
interesting one, and hope that the variety of J\'{o}nsson-Tarski algebras will 
be a useful data point. One other available data point is 
the variety of abelian groups, where the number of 
nonisomorphic J\'{o}nsson algebras was shown to be $\aleph_0$ 
by Scott in \cite{scott}. We tend to think of J\'{o}nsson algebras 
as being rare, and the variety of abelian groups seems to 
confirm
this intuition, but the variety of J\'{o}nsson-Tarski algebras provides an
instance where a variety 
contains the maximum theoretically possible number of 
nonisomorphic J\'{o}nsson algebras, so 
the intuition of J\'{o}nsson algebras as rare cannot be made 
rigorous in the 
obvious way.

\section{J\'{o}nsson Algebras in Subalgebra 
Distributive Varieties.}\label{subdist}

In this section we show that a sufficiently large algebra 
with a distributive subalgebra 
lattice always has a proper subalgebra of the same cardinality 
as the whole algebra. It will follow that algebras of cardinality 
greater than $\aleph_1$ in the variety of J\'{o}nsson-Tarski algebras
cannot be J\'{o}nsson.

\begin{lm} \label{BminusAsub}
    Let $J$ be an algebra whose subalgebra lattice is 
    distributive, and let $A \leq B \leq J$. If $S$ is a 
    subset of $J$ such that $\langle s \rangle \cap (B 
    \setminus A) = \emptyset$ for all $s \in S$, then it
    follows that $\langle S \rangle \cap (B \setminus A) 
    = \emptyset$. 
\end{lm}

\begin{proof}
    If the subalgebra lattice of $J$ satisfies the distributive 
    law, then it also satisfies the following infinite version 
    of the distributive law:
    \begin{equation*}\tag*{$(\ast)$}
    H \wedge 
    \left( \bigvee_{i \in I} K_i \right) 
     = \bigvee_{i \in I} (H \wedge K_i).
    \end{equation*}

    \noindent This is because the subalgebra lattice of $J$
    is an algebraic lattice, and algebraic lattices are 
    meet-continuous, meaning that binary meet distributes 
    over 
    up-directed joins. After rewriting $\bigvee_{i \in I} K_i$
    above as the join of the up-directed set consisting of 
    finite joins of the $K_i$'s, we can apply the meet-continuous 
    property and then apply the distributive law to prove 
    equation $(\ast)$.

    Now with $A$, $B$, 
    and $S$ as in the statement of the lemma, we get 
     \begin{align*}
      B \wedge \langle S \rangle &= 
      B \wedge \left( \bigvee_{s \in S} \langle s \rangle 
      \right) \\
      & = \bigvee_{s \in S} (B \wedge \langle s \rangle )
     \end{align*}
    \noindent where the second equality follows by 
    equation $(\ast)$.

    The assumption that $\langle s \rangle \cap (B 
    \setminus A) = \emptyset$ for all $s \in S$ 
    means that each $(B \wedge 
    \langle s \rangle)$ is contained in $A$, so the join 
    of those terms will be contained in $A$ as well. Thus 
    the above equality gives $B \wedge \langle S \rangle 
    \subseteq A$, from which 
    $\langle S \rangle \cap (B \setminus A) = \emptyset$ 
    follows.
\end{proof}

We will also make use of a set-theoretic result of A. Hajnal.
We use the standard notation 
$[\kappa]^{< \mu}$,
where $\kappa$ and $\mu$ are cardinals, 
to denote the set of subsets of $\kappa$ of size 
less than $\mu$.

\begin{df}
For a set $X$, a \textbf{set-mapping} on $X$
is a function $f: X \to \mathcal{P}(X)$
such that 
$x \not \in f(x)$ for all $x \in X$. 
When $f$ is a set-mapping on $X$, a subset 
$Y \subseteq X$ is called an 
$f$-\textbf{free set} if $a \not \in f(b)$ for all
$a, b$ in $Y$.
\end{df}

\begin{thm}\label{hajnal}
\begin{sloppypar}
{\bf (Theorem 1 from \cite{hajnal})}
If $\kappa > \mu$ are cardinals, $\kappa$ is infinite, and 
${f: \kappa \to [\kappa]^{<\mu}}$, then there is an 
$f$-free set $I \subseteq \kappa$ of cardinality 
$\kappa$.
\end{sloppypar}
\end{thm}

Now we are ready to prove the main theorem of the section.

\begin{thm}\label{aleph1max}
    Let $J$ be an algebra of cardinality $\kappa$ in a 
    language of cardinality $\lambda$. If $\kappa > \lambda^+$
    and the 
    subalgebra lattice of $J$ is distributive,
    then $J$ has a proper subalgebra of cardinality $\kappa$.
\end{thm}

\begin{proof}
Let $\{j_\alpha\}_{\alpha < \kappa}$ enumerate 
the elements of $J$. Then define a strictly
increasing 
$\kappa$-sequence
of subalgebras, $\{ J_\alpha \}_{\alpha < \kappa}$, such that 
$\bigcup_{\alpha < \kappa} J_\alpha = J$,
and such that $|J_\alpha| < \kappa$ for all $\alpha < \kappa$.

(Here is one way to accomplish this: for each 
$\alpha < \kappa$ define 
$H_{\alpha} \defeq 
\langle \{j_\beta : \beta \leq \alpha\} \rangle$. 
Then $\{H_\alpha\}_{\alpha < \kappa}$ is an 
increasing sequence of subalgebras whose union is $J$, 
and we have $|H_{\alpha}| \leq \lambda 
\cdot |\alpha + 1| < \kappa$ for each $\alpha$. Finally the 
sequence $\{H_\alpha\}_{\alpha < \kappa}$ may not be 
{\it strictly} increasing, but we can extract a strictly 
increasing 
subsequence $\{J_\alpha\}_{\alpha < \kappa}$, which now 
satisfies all our desired criteria.)

\smallskip 

Now define the set-mapping 
$f: \kappa \to [\kappa]^{< \lambda^+}$ by 

\[\alpha \mapsto \{\beta \neq \alpha: \langle j_\alpha \rangle \cap
    (J_{\beta+1}
    \setminus J_\beta) \neq \emptyset \}.\]

\noindent Since the size of the language is $\lambda$, we have 
$|\langle j_\alpha \rangle| \leq \lambda$ for all $\alpha$,
so $\langle j_\alpha \rangle$ can only 
intersect at most $\lambda$ many of the non-overlapping 
$(J_{\beta+1} \setminus J_\beta)$'s. This ensures that $f$
does indeed map $\kappa \to [\kappa]^{< \lambda^+}$.
We can now apply Theorem \ref{hajnal} with $\mu = \lambda^+$ 
to obtain an $f$-free set 
$I \subseteq \kappa$ of cardinality $\kappa$. 

\begin{clm}\label{propersubalg}
Let $S \defeq \{j_\alpha: \alpha \in I \setminus X\}$, 
where $X$ is any nonempty 
subset of $I$. Then $\langle S \rangle$ is a proper 
subalgebra of $J$.
\end{clm}
\begin{cpf}
Let $\xi \in X$ be fixed. Then for each $j_\alpha \in S$, 
we know that $\xi \neq \alpha$ (by the definition of $S$), 
and we also know that 
$\xi \not \in f(\alpha) = 
\{\beta \neq \alpha: \langle j_\alpha \rangle \cap
(J_{\beta+1}
\setminus J_\beta) \neq \emptyset \}$ because $I$ is $f$-free.
Therefore
we must have 
$\langle j_\alpha \rangle \cap (J_{\xi + 1} \setminus J_{\xi})
= \emptyset$. This reasoning applies to each $j_\alpha \in S$,
so Lemma \ref{BminusAsub} implies that $\langle S \rangle \cap
(J_{\xi+1} \setminus J_\xi) = \emptyset$, and 
$J_{\xi+1} \setminus J_\xi$ is nonempty because the sequence 
$\{J_\alpha\}_{\alpha < \kappa}$ is strictly increasing. 
Thus $\langle S \rangle$ is a proper subalgebra of $J$.
\end{cpf}




\smallskip

To complete the proof of the theorem, choose any $\xi \in I$
and let $S \defeq \{j_\alpha: \alpha \in I \setminus 
\{\xi\} \}$. Then $|\langle S \rangle| \geq |S| 
= |I| = \kappa$, and
$\langle S \rangle$
is a proper subalgebra of $J$ by Claim 
\ref{propersubalg}.
\end{proof}

Before connecting Theorem \ref{aleph1max} to our study 
of J\'{o}nsson J\'{o}nsson-Tarski algebras, we present 
another result that uses similar ideas
to those of Theorem \ref{aleph1max}. We will make use of 
the following theorem of G. Fodor:

\begin{thm}\label{fodor}
{\bf (Theorem 1 from \cite{fodor})}
If $\kappa > \mu$ are cardinals, $\mu$ is infinite, 
and \\ $f: \kappa \to [\kappa]^{<\mu}$ is a set-mapping, 
then $\kappa$ is the union of $\mu$ many $f$-free sets.
\end{thm}

\noindent The theorem in \cite{fodor} is slightly more general 
and 
phrased in terms of relations, not set-mappings; to see that it 
implies 
the version stated here, let $R$
be the relation defined by $xRy \leftrightarrow x \in f(y)$.

Fodor's result combined with the ideas of Theorem \ref{aleph1max}
yields the following theorem:

\begin{thm}\label{union}
Let $J$ be an algebra of cardinality $\kappa$ in a 
language of cardinality $\lambda$. If $\kappa > \lambda^+$
and the 
subalgebra lattice of $J$ is distributive,
then $J$ is the union of $\lambda^+$ many proper subalgebras.
\end{thm}

\begin{proof}
Define the sequences $\{j_\alpha\}_{\alpha < \kappa}$, 
$\{J_\alpha\}_{\alpha < \kappa}$, and the set-mapping 
\[f: \kappa \to [\kappa]^{<\lambda^+}:
\alpha \mapsto \{\beta \neq \alpha: \langle j_\alpha \rangle \cap
    (J_{\beta+1}
    \setminus J_\beta) \neq \emptyset \}\] as in Theorem 
\ref{aleph1max}.

Applying Theorem \ref{fodor} with $\mu = \lambda^+$, we get that 
$\kappa = \bigcup_{\alpha < \lambda^+} I_\alpha$, where 
each $I_\alpha$ is an $f$-free set. We will now argue that, for 
each $\alpha < \kappa$, 
the set $S_\alpha \defeq \{j_\beta: \beta \in I_\alpha\}$ is 
contained within 
the union of countably many proper subalgebras. This suffices 
to prove the theorem, since all of $J$ is contained within 
the union of the $S_\alpha$'s.

Let $\alpha$ be fixed, and select a countably infinite subset 
$\{i_0, i_1, i_2, \ldots \} \subseteq I_\alpha$. 
Consider the family of sets $\{S_{\alpha, n}\}_{n < \omega}$
given by
\[ S_{\alpha, n} \defeq \{j_\beta: \beta \in
I_\alpha \setminus \{i_m: m > n\} \}. \]

Now $\bigcup_{n < \omega} S_{\alpha, n} = S_\alpha$,
so $S_\alpha$ is contained within the union of the countable 
family of subalgebras $\{\langle 
  S_{\alpha, n} \rangle\}_{n \in \omega}$. 
To show that the subalgebras $\langle S_{\alpha, n} \rangle$
are proper, we can simply reuse Claim \ref{propersubalg} from 
the previous theorem, with $I_\alpha$ in place of $I$.
This finishes the proof.
\end{proof}

Now we will clarify the connection to J\'{o}nsson 
J\'{o}nsson-Tarski algebras. 
A {\bf subalgebra distributive variety} is a 
variety in which 
all algebras have distributive subalgebra lattices. 
So the following 
corollary is immediate from Theorem \ref{aleph1max} and our 
definition that J\'{o}nsson algebras have a countable 
language:

\begin{cor}\label{subdistmaxcard}
A subalgebra distributive variety
cannot contain J\'{o}nsson algebras 
with cardinality greater than $\aleph_1$.
\end{cor}

\noindent And now we can fully answer the question of which 
cardinalities of J\'{o}nsson J\'{o}nsson-Tarski algebras are possible.

\begin{cor}\label{JTaleph1}
The variety of J\'{o}nsson-Tarski algebras does not contain J\'{o}nsson algebras 
of cardinality greater than $\aleph_1$.
\end{cor}

\noindent To prove Corollary \ref{JTaleph1} it suffices to show that the 
variety of J\'{o}nsson-Tarski algebras is subalgebra distributive. We recall 
two results here: the first is Lemma 3.1 from \cite{dubeau-kearnes},
which states:

\begin{lm}\label{muterm}
{\bf (Lemma 3.1 from \cite{dubeau-kearnes})} 
Every $\mathcal{L}$-term is $\Sigma$-equivalent to a term in 
$\mathcal{MU}$.
\end{lm}

Here, $\mathcal{L}$ is the language of J\'{o}nsson-Tarski algebras, 
$\Sigma$ is the defining set of identities for J\'{o}nsson-Tarski algebras, 
and $\mathcal{MU}$ is the set of all $m,u$-terms: terms of the 
form $m(u_1(x_{1}), \dots, u_k(x_{k}))$, where $m$ is a 
multiplicative term (using $\cdot$ only) and $u_1, \dots, u_k$
are unary terms (using $\ell$ and $r$ only).

We also recall Theorem 4.3 from \cite{shapiro2}, attributed 
to R. McKenzie:

\begin{thm}\label{shapirothm}
{\bf (Theorem 4.3 from \cite{shapiro2})}

A variety $\mathcal{V}$ is subalgebra distributive 
if and only if for each term 
$p(\bar{x}, \bar{y})$ there exists a term $s$ and unary 
terms $u_1, \dots, u_k, v_1, \dots, v_\ell$, where 
$k$ and $\ell$ are nonnegative integers, such that the 
following are identities of $\mathcal{V}$:
\[p(\bar{x}, \bar{y}) = s\big( \, u_1(p(\bar{x}, \bar{y})), 
u_2(p(\bar{x}, \bar{y})), \dots, u_k(p(\bar{x}, \bar{y})),
v_1(p(\bar{x}, \bar{y})), \dots, v_\ell(p(\bar{x}, \bar{y})) 
\, \big)\]

\noindent and $u_i(p(\bar{x}, \bar{y})) = u_i(p(\bar{x}, \bar{z}))$
for $0 \leq i \leq k$, $v_i(p(\bar{x}, \bar{y})) = 
v_i(p(\bar{z}, \bar{y}))$ for $0 \leq i \leq \ell$.
\end{thm}

\smallskip

\noindent{\it Proof of Corollary \ref{JTaleph1}.}\
We will argue that the condition of Theorem \ref{shapirothm}
is satisfied in the variety $\mathcal{V}$
of J\'{o}nsson-Tarski algebras.

Let $p(\bar{x},
\bar{y})$ be a term. By Lemma \ref{muterm},
$p(\bar{x},\bar{y})$ is $\Sigma$-equivalent to an
$m,u$-term, say, $p'(\bar{x},\bar{y})$. Write $p'(\bar{x},\bar{y}) = 
m\big( \, w_1(\bar{x},\bar{y}), \dots, w_n(\bar{x},\bar{y})
\, \big)$, where $m$ is an $\mathcal{L}_m$
term and each $w_i$ is 
unary.

\begin{clm}\label{unaries}
To each $w_i$ there corresponds a 
unary term $t_i$ such that $\Sigma$ entails the identity 
$t_i(p'(\bar{x},\bar{y})) = w_i(\bar{x},\bar{y})$. 
\end{clm}

\begin{cpf}
We argue for a fixed $i \leq n$ by induction on the 
complexity of $m$. If $m$ is simply a variable, then 
$t_i$ can be taken to be a variable. 
Otherwise $m = m_1 
\cdot m_2$, where $m_1$ and $m_2$ are $\mathcal{L}_m$ terms, 
and therefore we can write
\begin{align*}
p'(\bar{x},\bar{y}) &= m\big( \, w_1(\bar{x},\bar{y}), \dots, 
w_n(\bar{x},\bar{y})
\, \big) \\
&=
m_1 \big( \, w_{i_1}(\bar{x},\bar{y}), \dots, w_{i_p}(\bar{x},\bar{y})
\, \big) \cdot m_2 \big( \, w_{j_1}(\bar{x},\bar{y}), \dots, 
w_{j_q}(\bar{x},\bar{y})
\, \big)
\end{align*}

\noindent The term $w_i(\bar{x}, \bar{y})$ must appear in 
the final expression above as an argument to 
$m_1$ or $m_2$ (possibly both). Say it appears as an 
argument to $m_1$.
By inductive hypothesis there is a unary term $t(x)$ so that 
$\Sigma$ entails the identity 
$t \big(m_1 (w_{i_1}(\bar{x},\bar{y}), \dots, 
w_{i_p}(\bar{x},\bar{y})
) \big) = w_i(\bar{x}, \bar{y}).$ Therefore we let 
$t_i = t(\ell(x))$, so that $\Sigma$ entails 
\[t_i (p'(\bar{x}, \bar{y})) = t(\ell(p'(\bar{x},\bar{y})))
= t \big(m_1 (w_{i_1}(\bar{x},\bar{y}), \dots, 
w_{i_p}(\bar{x},\bar{y})
) \big) =
w_i(\bar{x}, \bar{y})\]
as desired.
The case where $w_i$ appears as an argument to 
$m_2$ is similar, using $r$ instead of $\ell$.
\hfill \end{cpf}


Now Claim \ref{unaries} gives us that the following is an 
identity of $\mathcal{V}$:
\[p'(\bar{x},\bar{y}) = 
m\big( \, t_1(p'(\bar{x}, \bar{y})), 
\dots, t_n(p'(\bar{x}, \bar{y})) \, \big),\] 
where each $t_i$ is unary. Since $\Sigma$ entails the 
identity $p(\bar{x},\bar{y})
= p'(\bar{x},\bar{y})$, the following is also an identity of 
$\mathcal{V}$:
\[p(\bar{x},\bar{y}) = 
m\big( \, t_1(p(\bar{x}, \bar{y})), 
\dots, t_n(p(\bar{x}, \bar{y})) \, \big),\] 
This yields the desired identities found in 
Theorem \ref{shapirothm} by
relabeling the $t_i$'s as either $u_i$'s or $v_i$'s, according 
to whether the $t_i$'s depend on a variable from $\bar{x}$
or $\bar{y}$.

We have now proven that the variety of J\'{o}nsson-Tarski algebras 
is subalgebra 
distributive, which finishes the proof of 
Corollary \ref{JTaleph1}. \qed

\smallskip

Now the possible cardinalities of J\'{o}nsson J\'{o}nsson-Tarski 
algebras are completely determined. It is worth mentioning 
that these results have a 
purely combinatorial interpretation:

\begin{cor}
The following statement holds if and only if 
$\kappa \geq \aleph_2$:

For every set $S$ of cardinality $\kappa$, and every bijection 
$f: S \to S \times S$, there exists a subset $T \subsetneq S$ 
of cardinality $\kappa$ such that $f \restriction T$ is a 
bijection $T \to T \times T$.
\end{cor}

\begin{proof}
The statement above is equivalent to the statement ``there does 
not exist a 
J\'{o}nsson J\'{o}nsson-Tarski algebra of cardinality $\kappa$.''
Hence Theorems 4.1 and 4.4 from \cite{dubeau-kearnes} prove
the statement false when $\kappa = \aleph_0$ and $\aleph_1$, 
and Corollary \ref{JTaleph1} proves the statement true for 
$\kappa \geq \aleph_2$.
\end{proof}
We also suggest a possible 
generalization of Corollary \ref{subdistmaxcard}:

\begin{quest}
Can a strongly abelian variety in a countable language 
contain a J\'{o}nsson algebra 
with cardinality greater than $\aleph_1$?
\end{quest}

\section{J\'{o}nsson Algebras in Residually Small Varieties.}\label{resid}
In \cite{dubeau-kearnes} the authors showed that a residually finite 
J\'{o}nsson algebra 
cannot exist in a minimal variety. An important part of this work 
was the following corollary:

\begin{cor}\label{residfinite}
{\bf (Corollary 2.2 from \cite{dubeau-kearnes})}
Let $J$ be a J\'{o}nsson algebra which has a finite bound $n$ on 
the size of its cyclic subalgebras. $J$ is not residually 
finite.
\end{cor}

The main theorem of this section is a generalization and 
extension of Corollary \ref{residfinite}, with a 
similar proof. We phrase it as follows:

\begin{thm}\label{residsmall}
Let $J$ be an algebra of cardinality $\kappa$ in a language 
of cardinality $\lambda$. If $\cf(\kappa) > 2^{\lambda^+}$ 
and $J$
lies in a  
residually small variety,
then $J$ has a proper subalgebra of size $\kappa$.
\end{thm}



Note that all the assumptions of Corollary \ref{residfinite}
are relaxed in the statement of Theorem \ref{residsmall}.
The cardinality of the 
language is replaced with an arbitrary $\lambda$, which in 
turn means that the cyclic subalgebras of $J$ will have 
cardinality less than $\lambda^+$; by assuming that $J$ 
lies in a residually small variety, we are effectively 
assuming that $J$ is residually less than $(2^\lambda)^+$
(see Theorem 1.2 of \cite{taylor}). 
The effect of relaxing the assumptions
is that we must use one more assumption, 
namely that $\cf(\kappa) > 2^{\lambda^+}$, to conclude 
that $J$ has a proper subalgebra of size $\kappa$.



We will now prove Theorem \ref{residsmall}. 
We first recall Theorem 2.1(2) from 
\cite{dubeau-kearnes}:

\begin{thm}\label{smallquotient}
{\bf (Theorem 2.1(2) from \cite{dubeau-kearnes})}
Let $J$ be an algebra with no proper subalgebras of cardinality 
$|J|$, and 
let $\theta$ be a conguence on $J$.
If $| J/\theta | < \cf(|J|)$, then 
$J/\theta$ is cyclic.
\end{thm}

We reformulated Theorem \ref{smallquotient} slightly: in 
\cite{dubeau-kearnes}, it was assumed that $J$ was a
J\'{o}nsson algebra. The only difference here is that we do not 
assume the language of $J$ is countable. This assumption 
was never used in the proof, so our reformulated version is still 
valid. 

\smallskip

\noindent \textit{Proof of Theorem \ref{residsmall}.}

\begin{sloppypar}
Let $J$ be an algebra of cardinality $\kappa$ in a language of 
cardinality $\lambda$. Suppose that ${\cf(\kappa) > 
2^{\lambda^+}}$
and that $J$ lies in a
residually small variety.
Choose a subset $X = \{x_\alpha\}_{\alpha < \lambda^+}$ 
of $J$. 
\end{sloppypar}

Then, for each pair $\alpha, \beta < \lambda^+$, choose a 
congruence 
$\theta_{\alpha, \beta}$ of $J$ which is maximal for the property 
that $(x_\alpha, x_\beta) \not \in \theta_{\alpha, \beta}$.
It follows that for each pair $\alpha, \beta$, the quotient 
algebra 
$J / \theta_{\alpha,\beta}$ is subdirectly irreducible, since 
the congruence lattice of $J/\theta_{\alpha,\beta}$ contains 
a monolith, namely the smallest congruence relating $x_\alpha / 
\theta_{\alpha,\beta}$ 
and $x_\beta / \theta_{\alpha,\beta}$.

The fact that $J$ lies in a residually small 
variety means that these 
subdirectly irreducible quotients of $J$ have size 
$\leq 2^{\lambda}$ (see Theorem 1.2 of \cite{taylor}).
Moreover, defining $\displaystyle \theta \defeq 
\bigcap_{\alpha,\beta < \lambda^+} \theta_{\alpha,\beta}$, 
we see that $J/\theta$ is subdirectly embeddable in 
the product of the $J/\theta_{\alpha,\beta}$'s, giving the 
following cardinality bound:

\[ |J/\theta| \leq \left| \prod_{\alpha,\beta < \lambda^+} 
J/\theta_{\alpha,\beta} \right| \leq (2^{\lambda})^{\lambda^+ 
\times \lambda^+} = 2^{\lambda^+}. \]

Now, if we suppose for contradiction that $J$ has 
no proper subalgebras of cardinality $\kappa$, then Theorem 
\ref{smallquotient} would
imply that $J/\theta$ is cyclic, and 
therefore $|J/\theta| \leq \lambda$.
However, no two $x_\alpha$'s occupy the same 
$\theta$-class, so the cardinality of $J/\theta$ must 
be at least $\lambda^+$. 
Thus $J$ must have a proper subalgebra of cardinality 
$\kappa$. \qed

We end the section with some remarks about Theorem \ref{residsmall}.
First, the theorem has strong implications for the 
construction of large J\'{o}nsson algebras in ZFC.
In particular, if ZFC is consistent,
then one cannot construct a J\'{o}nsson algebra in a residually 
small variety of size $\kappa = \aleph_n$ 
for any finite
$n \geq 3$. This is because Theorem \ref{residsmall}, together 
with the fact that J\'{o}nsson algebras have countable languages 
by definition, 
implies that a J\'{o}nsson algebra of size $\kappa$ in a 
residually small variety must have 
$\cf(\kappa) \leq 2^{\aleph_1}$, and
if ZFC is consistent, then 
there exist models of ZFC in which $2^{\aleph_1} = \aleph_2$
(e.g. any model in which GCH holds). In fact, any 
J\'{o}nsson algebra in a residually small variety constructed in ZFC 
must either 
have cardinality $\leq \aleph_2$, or it must have singular 
cardinality, and J\'{o}nsson algebras of singular cardinality are 
apparently very difficult to construct. Currently there are no known 
examples.

We also remark that Theorem \ref{residsmall} is connected to our study 
of J\'{o}nsson J\'{o}nsson-Tarski algebras, 
since the 
variety of J\'{o}nsson-Tarski algebras is residually small. We showed 
in Corollary \ref{JTaleph1} that the variety of J\'{o}nsson-Tarski algebras 
is subalgebra distributive; Lemma 2 from \cite{shapiro1} 
proves that subalgebra distributive varieties are strongly 
abelian; finally, an unpublished result announced by  
E. Kiss states that strongly abelian varieties are residually 
small. So Theorem \ref{residsmall} implies that a J\'{o}nsson 
J\'{o}nsson-Tarski algebra of cardinality $\aleph_3$ cannot be 
constructed in ZFC, by the logic of the previous paragraph.
Ultimately Theorem \ref{residsmall} 
does not give as tight of a bound on the cardinality of 
J\'{o}nsson J\'{o}nsson-Tarski algebras as Corollary \ref{JTaleph1} does,
but the class of residually 
small varieties is much larger than the class of subalgebra 
distributive varieties, so we probably cannot expect as 
restrictive of a cardinality bound.

As far as we know, it is still possible that 
Theorem \ref{residsmall} could be improved:

\begin{quest}
Is there a maximum cardinality that a J\'{o}nsson 
algebra in a residually small variety can have?
\end{quest}

\section{More J\'{o}nsson J\'{o}nsson-Tarski Algebras.}\label{JJT}
In Theorem 4.4 of \cite{dubeau-kearnes}, a J\'{o}nsson J\'{o}nsson-Tarski algebra of 
cardinality $\aleph_1$ was constructed. In this section we prove,
by expanding on the construction from \cite{dubeau-kearnes}, that 
there exist $2^{\aleph_1}$ many nonisomorphic J\'{o}nsson 
J\'{o}nsson-Tarski algebras of cardinality $\aleph_1$. A nice feature of the 
expanded construction is that, unlike in \cite{dubeau-kearnes},
{\it any} J\'{o}nsson-Tarski algebra with universe $\omega$ can be used as the 
starting point. So we actually prove a stronger theorem:

\begin{thm}\label{manyjonsson}
Any countable J\'{o}nsson-Tarski algebra can be extended into 
$2^{\aleph_1}$ many pairwise nonisomorphic J\'{o}nsson J\'{o}nsson-Tarski algebras 
of cardinality $\aleph_1$.
\end{thm}

We will need to assume that the reader is familiar with the construction 
found in Theorem 4.4 of \cite{dubeau-kearnes}. The basic principle 
of the construction was to define the multiplication table 
as a sequence 
of what we will now call ``layers''. That is, the construction 
in Theorem 4.4 of \cite{dubeau-kearnes}
started with a particular
J\'{o}nsson-Tarski algebra with universe $\omega$, 
which was called $J_\omega$.
Then it was shown how, 
for any limit ordinal 
$\lambda < \omega_1$, one could extend the current 
J\'{o}nsson-Tarski algebra $J_\lambda$ with universe $\lambda$ into 
a J\'{o}nsson-Tarski algebra $J_{\lambda+\omega}$ with universe $\lambda+\omega$.
This meant adding the ordinals $\{\lambda + n: n \in \omega\}$ 
to the multiplication table of $J$, where they occupied the 
unshaded regions in Figure \ref{jtpicturesimple}.
We think of this inductive step as adding a new
``layer'' to the multiplication table.

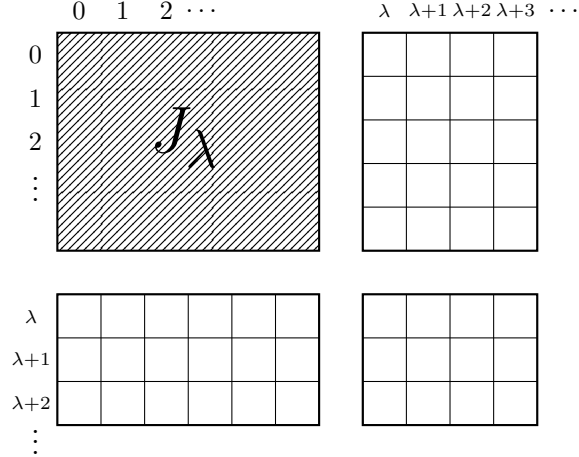
\begin{figure}[ht] 
  \begin{tikzpicture}[scale=0.58]
  \draw [thick, pattern = north east lines] (0,11) rectangle (6,6);

  \draw [thick] (0,5) rectangle (6,2);
  \draw [thick] (7,5) rectangle (11,2);
  \draw [thick] (7,11) rectangle (11,6);
  
  \foreach \y in {2,...,5}
  {
    \draw [thin] (7, \y) -- (11, \y);
    \draw [thin] (0, \y) -- (6, \y);
  }
  \foreach \y in {6,...,11}
  {
    \draw [thin] (7, \y) -- (11, \y);
  }
  \foreach \x in {7,...,11}
  {
    \draw [thin] (\x, 6) -- (\x, 11);
    \draw [thin] (\x, 2) -- (\x, 5);
  }
  \foreach \x in {0,...,7}
  {
    \draw [thin] (\x, 2) -- (\x, 5);
  }
  
  \node at (7.5, 11.5) {\tiny $\lambda$};
  \foreach \x in {0,1,2,}
    \node at (\x+.5, 11.5) {\small \x};
  \foreach \x in {8,...,10}
  {
    \pgfmathtruncatemacro{\n}{\x-7};
    \node at (\x+.5, 11.5) {\tiny $\lambda$+\n};
  }
  \foreach \y in {13,14,15}
  {
    \pgfmathtruncatemacro{\n}{15-\y};
    \node at (-0.5, \y - 4.5) {\small \n};
  }
  \node at (-0.6, 4.5) {\tiny $\lambda$};
  \foreach \y in {5,...,6}
  {
    \pgfmathtruncatemacro{\n}{7-\y};
    \node at (-0.6, \y - 2.5) {\tiny $\lambda$+\n};
  }
  
  \node at (3.3, 11.5) {\small $\ldots$};
  \node at (11.6, 11.5) {\small $\ldots$};
  \node at (-0.5, 7.6) {$\vdots$};
  \node at (-0.5, 1.8) {$\vdots$};

  \node at (2.9,8.6) {\Huge $J_\lambda$};
  \end{tikzpicture}
  \bigskip
  \caption{Construction of $J_{\lambda+\omega}$ from $J_\lambda.$
           The unshaded cells will be occupied by the ordinals 
           $\{\lambda + n: n \in \omega\}$. We refer to this 
           unshaded region of the table as a layer.}
  \label{jtpicturesimple}
  \end{figure}


Our new system of construction will have largely the same form,
except for two key differences. First, $J_\omega$ can be taken 
to be any J\'{o}nsson-Tarski algebra with universe $\omega$.
Second, at each stage of the construction,
we can choose to add either a ``type A'' layer or a ``type B'' 
layer. The {\bf type A layer} will be exactly the same design that 
was seen in the construction from \cite{dubeau-kearnes}. The 
{\bf type B layer} is defined to be the transpose, or mirror image,
of the type A
layer: exactly the same, except that $\ell$ and $r$ are 
exchanged. We will show that
one can choose any $\omega_1$-sequence of type A and type B layers 
and the resulting structure will be J\'{o}nsson, 
following largely the same 
argument as in \cite{dubeau-kearnes}. 
But we will argue that 
the type B layer is 
different 
enough so that there are $2^{\aleph_1}$ many sequences of 
type A and B layers which all produce pairwise nonisomorphic 
algebras.

\begin{df}
Denote by $L_{>0}$ the set of all nonzero countable limit ordinals.
For an arbitrary J\'{o}nsson-Tarski algebra $J_\omega$ with universe $\omega$,
and $\sigma: L_{>0} \to \{A,B\}$ an assignment of the letter $A$ or 
$B$ to each limit ordinal from $L_{>0}$, we define 
$J_\omega^\sigma$ to be 
the J\'{o}nsson-Tarski algebra constructed by extending $J_\omega$
according to $\sigma$: that is, 
$J_\omega^\sigma$ is formed by starting with $J_\omega$ and, 
for each $\lambda \in L_{>0}$, 
extending $J_\lambda$ to $J_{\lambda+\omega}$ using the 
type $\sigma(\lambda)$ construction described above 
(either type A or type B).
\end{df}

\begin{lm}\label{ABjonsson}
Let $J_\omega$ be an arbitrary J\'{o}nsson-Tarski algebra with universe $\omega$,
and let $\sigma: L_{>0} \to \{A,B\}$. 
Then $J_{\omega}^\sigma$ is J\'{o}nsson.
\end{lm}

\begin{proof}
We split the proof into four minor claims, each with a simple 
proof:
\begin{clm}\label{genlambda}
Let $\lambda \in L_{>0}$. For any $n \in \omega$, the element 
$\lambda+n$
generates the element $\lambda$.
\end{clm}
\begin{cpf}
If $\lambda+n$ resides in a type A layer, the proof is 
the same as in \cite{dubeau-kearnes}. In fact, in the proof from
\cite{dubeau-kearnes}, we saw that the element $\lambda+n$ 
generates the 
element $\lambda$ via a sequence of $\ell$'s and $r$'s which 
stayed within the set $\{\lambda + m: m \in \omega\}$. 
In the type B layer, where $\ell$ and $r$ behave on the set 
$\{\lambda + m: m \in \omega\}$ as $r$ and $\ell$ respectively
behaved 
in the type A layer, we can use the opposite 
sequence of $\ell$'s and $r$'s to generate $\lambda$ from 
$\lambda+n$.
\end{cpf}

\begin{clm}\label{lambdagen}
For any $\lambda \in L_{>0}$, the element $\lambda$ generates all of 
the set $\lambda + \omega$.
\end{clm}
\begin{cpf}
In the type A layer presented in \cite{dubeau-kearnes}, we saw that 
$\lambda$ can generate any element $x \in \lambda + \omega$ 
as $x = \ell(r^n (\lambda))$, for some $n \in \omega$, where 
each $r^k(\lambda)$ ($k \leq n$) lies within the set 
$\{\lambda + m: m \in \omega\}$.
This means that in the type 
B layer, $\lambda$ can generate any element $x \in \lambda+\omega$ 
as $x=r(\ell^n(\lambda))$.
\end{cpf}
\begin{clm}\label{alphagen}
For any $\omega \leq \alpha < \omega_1$, 
$\{\alpha\}$ 
generates all of the set $\alpha + \omega$.
\end{clm}
\begin{cpf}
Follows immediately from Claims \ref{genlambda} and \ref{lambdagen}:
$\alpha$ can be written as some $\lambda+n$, for $\lambda \in L_{>0}$ 
and $n \in \omega$. Then $\lambda+n$ generates $\lambda$, and 
$\lambda$ in turn generates the set $\lambda+\omega = \alpha+\omega$.
\end{cpf}

\begin{clm}\label{subalgebras}
The subalgebras of $J_\omega^\sigma$ are exactly the countable 
limit ordinals, together with any subalgebras of $J_\omega$.
\end{clm}
\begin{cpf}
The countable limit ordinals are subalgebras by construction, 
as are any subalgebras of $J_\omega$. The fact that there are 
no other subalgebras follows from Claim \ref{alphagen}.
\end{cpf}

Now Claim \ref{subalgebras} establishes that all subalgebras of 
$J_\omega^\sigma$ are countable, so $J_\omega^\sigma$ is 
J\'{o}nsson.
\end{proof}

\begin{lm}\label{ABnonisomorphic}
Suppose $J_{\lambda_1}$ and $J_{\lambda_2}$ are arbitrary 
J\'{o}nsson-Tarski algebras with universes $\lambda_1, \,  
\lambda_2 \in L_{>0}$.
Let $J_{\lambda_1+\omega}^A$ denote
$J_{\lambda_1}$ extended 
with a type A layer,
and let $J_{\lambda_2+\omega}^B$ denote $J_{\lambda_2}$ 
extended with 
a type B layer. Then $J_{\lambda_1+\omega}^A$ and 
$J_{\lambda_2+\omega}^B$ are not isomorphic.
\end{lm}

\begin{proof}
For a J\'{o}nsson-Tarski algebra $X$, define $g \in X$ to be a \textit{type B 
generator of $X$} if every $x \in X$ can be written as $r(\ell^n(g))$ 
for some $n \in \omega$. The existence of such an element within 
a J\'{o}nsson-Tarski algebra will be preserved under isomorphism. Thus
the following two claims 
establish that $J_{\lambda_1+\omega}^A$ and 
$J_{\lambda_2+\omega}^B$ are not isomorphic:

\begin{enumerate}
\item $J_{\lambda_2+\omega}^B$ contains a type B generator.
\item $J_{\lambda_1+\omega}^A$ does not contain a type B generator.
\end{enumerate}

To prove (1), we submit that the element $\lambda_2$ is a type 
B generator of 
$J_{\lambda_2+\omega}^B$. The reasoning is identical to that 
of Claim \ref{lambdagen} above.

To prove (2), we first note that 
in the type $A$ construction, we have $\ell(\lambda_1+n) < 
\lambda_1+n$
for all $n \in \omega$. For even ordinals $\lambda_1+2k$, this is 
clear; as for the odd ordinals $\lambda_1+2k+1$, recall that 
$\lambda_1+2k+1$ belongs to some set $L_{\lambda_1+m}$, $m<2k+1$,
and thus lies in the L-shaped region corresponding to 
$\lambda_1+m$
(see Figure 2 from \cite{dubeau-kearnes}). 
Applying $\ell$ to any member of 
that L-shaped region will produce at most $\lambda_1+m$, so we have 
that $\ell(\lambda_1+2k+1) < \lambda_1+2k+1$ as desired.

We also have, in the type $A$ construction, that $r(\lambda_1+n) 
\leq \lambda_1+n+2$ for all $n \in \omega$. When $n$ is even, 
$r(\lambda_1+n) = \lambda_1+n+2$ by definition; when $n$ is odd,
$r(\lambda_1+n) < \lambda_1+n$ as in the previous paragraph.

Now start with any $\lambda_1+m$ in the type $A$ construction 
and consider the sequence 
\[ \lambda_1+m, \, \ell(\lambda_1+m), \, \ell(\ell(\lambda_1+m)), 
\, \ell(\ell(\ell(\lambda_1+m))),
\ldots\]
For as long as this sequence stays within $J_{\lambda_1+\omega}^A 
\setminus J_{\lambda_1}$, it is decreasing; once the sequence 
enters $J_{\lambda_1}$, it will forever stay within 
$J_{\lambda_1}$, 
since $J_{\lambda_1}$ is a subalgebra and closed under $\ell$.
This proves that no member of the above sequence is greater 
than $\lambda_1+m$. Using that fact and the fact that 
$r(\lambda_1+n) \leq \lambda_1+n+2$ for all $n \in \omega$, 
it follows that any element of the form $r(\ell^n(\lambda_1+m))$ 
is at most $\lambda_1+m+2$. So there will be elements of 
$J_{\lambda_1+\omega}^A$ that are not of this form. 

We have now proven that no element of the form $\lambda_1+m \in 
J_{\lambda_1+\omega}^A$ is a type B generator of 
$J_{\lambda_1+\omega}^A$. The only other elements of 
$J_{\lambda_1+\omega}^A$ that must be checked are those elements 
less than $\lambda_1$. But such elements belong to $J_{\lambda_1}$,
a proper subalgebra of $J_{\lambda_1+\omega}^A$, so they cannot
possibly generate all 
of $J_{\lambda_1+\omega}^A$, and in particular they are 
not type B generators.

Now $J_{\lambda_2+\omega}^B$ contains a type B generator while 
$J_{\lambda_1+\omega}^A$ does not; the two algebras are not 
isomorphic.
\end{proof}

\begin{lm}\label{differentseq}
Let $J_\omega$ be an arbitrary J\'{o}nsson-Tarski algebra with universe $\omega$.
Let $\sigma_1: L_{>0} \to \{A,B\}$ and $\sigma_2: L_{>0} \to \{A,B\}$ 
be sequences satisfying the following criteria:
\begin{itemize}
\item $\sigma_1$ and $\sigma_2$ both begin with two occurrences
 of $A$, then one occurrence of $B$,
\item if $\sigma_1^-$ and $\sigma_2^-$ denote $\sigma_1$ and 
 $\sigma_2$ without their initial three elements, then neither
 $\sigma_1^-$ nor $\sigma_2^-$ has two consecutive 
 occurrences of $A$, and
\item $\sigma_1 \neq \sigma_2$.
\end{itemize}
Then the resulting J\'{o}nsson-Tarski algebras 
$J_\omega^{\sigma_1}$ and $J_\omega^{\sigma_2}$ are not isomorphic.
\end{lm}
\begin{proof}
We will use $J_\lambda^1$ for $\lambda \in 
L_{>0}$ to denote 
the subalgebra with universe $\lambda$ within 
$J_\omega^{\sigma_1}$.
For the subalgebra with universe $\lambda$ within 
$J_\omega^{\sigma_2}$, we use $J_\lambda^2$.

Suppose for contradiction that we have an 
isomorphism $\varphi: J_\omega^{\sigma_1} \to J_\omega^{\sigma_2}$.
Then $\varphi$ induces a lattice isomorphism $\psi:
\text{Sub}(J_\omega^{\sigma_1}) \to \text{Sub}(J_\omega^{\sigma_2})$.

\begin{clm}
  $\psi(J_{\omega}^1) = 
  J_{\omega}^2$.
\end{clm} 
\begin{cpf}
  If the equality does not hold, then we can assume 
  $\psi(J_{\omega}^1) > J_{\omega}^2$ without loss 
  of generality, as this is either the case for $\psi$ or $\psi^{-1}$
  and we can replace $\psi$ with $\psi^{-1}$ in the following 
  argument if necessary.
  
  Claim \ref{subalgebras} established that the portion of 
  $\text{Sub}(J_\omega^{\sigma_2})$ above $J_\omega^2$ is an 
  $\omega_1$-chain whose elements are exactly the nonzero 
  countable limit ordinals. So our assumption that 
  $\psi(J_{\omega}^1) > J_{\omega}^2$ means that 
  $\psi(J_{\omega}^1) = J_{\omega \cdot \alpha}^2$ for some 
  $\alpha > 1$. Now we
  compare the sequence fragments
  \[ J_{\omega \cdot 1}^1 \leq J_{\omega \cdot 2}^1 \leq
  J_{\omega \cdot 3}^1\]
  and 
  \[ J_{\omega \cdot \alpha}^2 \leq J_{\omega \cdot (\alpha+1)}^2 \leq
  J_{\omega \cdot (\alpha+2)}^2\]
  As $\psi$ is a lattice isomorphism, we must have 
  $\psi(J_{\omega \cdot (1+i)}^1) = J_{\omega \cdot (\alpha + i)}^2$
  for $i = 0,1,2$.
  But at least 
  one $J_{\omega \cdot (\alpha + i)}^2$, where $i=1\text{ or }2$,
  must be a type B extension of 
  its predecessor: if $\alpha=2$ or $3$, we use the 
  fact that $J_{\omega \cdot 4}^2$ is a type B extension of 
  $J_{\omega \cdot 3}^2$, and if $\alpha > 3$, we use
  the fact that 
  $\sigma_2^-$ 
  does not contain any instances of two consecutive A's. 
  Meanwhile the 
  corresponding $J_{\omega \cdot (i+1)}^1$ is a type A 
  extension of its predecessor, so the two cannot be isomorphic 
  according to Lemma \ref{ABnonisomorphic}, a contradiction.
\end{cpf}

Now that we know $\psi(J_{\omega}^1) = 
J_{\omega}^2$, it follows
that $\psi(J_\lambda^1) = J_\lambda^2$ for all $\lambda \geq 
\omega$, because $\psi$ is a lattice isomorphism 
and thus preserves successors and limits. But now we reach 
a contradiction at the first index where $\sigma_1$ and 
$\sigma_2$ differ, for one indicates a type A extension while 
the other indicates a type B extension, violating 
Lemma \ref{ABnonisomorphic}. This completes the 
proof.
\end{proof}

Now we can finish proving the main theorem of the section:

\noindent \textit{Proof of Theorem \ref{manyjonsson}}.

Lemmas \ref{ABjonsson} and \ref{differentseq} have 
done the majority of the work; we just need 
to clean up some loose ends. First, 
we should prove that there are $2^{\aleph_1}$ many 
different $\omega_1$-sequences of the letters A and B 
which do not contain two consecutive occurrences of A.
Let $S$ be the set of such sequences. 

We can create an 
injective map $\{0,1\}^{\omega_1} \to S$. To do this, 
take a sequence $\delta \in \{0,1\}^{\omega_1}$ and replace 
each occurrence of 0 with AB, and replace each occurrence 
of 1 with B. The result is a unique sequence in $S$ corresponding 
to each member of $\{0,1\}^{\omega_1}$, which shows that 
$|S| = 2^{\aleph_1}$.

Finally, we have now shown that any J\'{o}nsson-Tarski 
algebra {\it with universe} $\omega$ can be extended 
into $2^{\aleph_1}$ many pairwise nonisomorphic J\'{o}nsson 
J\'{o}nsson-Tarski algebras of cardinality $\aleph_1$. It follows 
that any countable J\'{o}nsson-Tarski algebra can be extended 
in the same way, by simply using a bijection between its universe 
and $\omega$ to relabel the elements before starting the 
construction.
\qed

\bigskip

\noindent {\bf Acknowledgment.}
We thank K. Kearnes for many helpful conversations and ideas.
We also thank the anonymous referee for simplifying the proof 
of Theorem \ref{aleph1max}, and for suggesting Theorem 
\ref{union} and its proof.

\bibliographystyle{amsplain}

\begin{thebibliography}{10}

\bibitem{coleman}
  Coleman, Eoin,
  {\it J\'{o}nsson groups, rings and algebras.}
  Irish Math.\ Soc.\ Bull.\ No.\ {\bf 36} (1996), 34--45


\bibitem{dubeau-kearnes}
DuBeau, Jordan; Kearnes, Keith A.,
 {\it An uncountable Jónsson algebra in a minimal variety.}
Proc.\ Amer.\ Math.\ Soc.\ {\bf 148} (2020), no. 4, 1765--1775

\bibitem{erdos-hajnal}
  Erd\H{o}s, Paul; Hajnal, Andras,
{\it On a problem of B. J\'{o}nsson.}
Bull.\ Acad.\ Polon.\ Sci.\ Sér.\ Sci.\ Math.\ Astronom.\ Phys.\
{\bf 14} (1966), 19--23. 

\bibitem{fodor}
  Fodor, G.,
{\it Proof of a conjecture of P. Erd\H{o}s.}
Acta Sci.\ Math.\ (Szeged)
{\bf 14} (1952), 219--227.

\bibitem{hajnal}
  Hajnal, Andras,
{\it Proof of a conjecture of S. Ruziewicz.}
Fund.\ Math.\ {\bf 50} (1961/62), 123--18.



\bibitem{jonsson-tarski}
  Jónsson, Bjarni; Tarski, Alfred,
  {\it On two properties of free algebras.} 
  Math. Scand. {\bf 9} (1961), 95--101.



  
  


\bibitem{scott}
  Scott, William R.,
  {\it Groups and cardinal numbers.}
  Amer.\ J.\ Math.\ {\bf 74}, (1952). 187--197. 

\bibitem{shapiro1}
  Shapiro, J.,
  {\it Finite equational bases for subalgebra distributive varieties.}
  Algebra Universalis {\bf 24}, (1987), no. 1-2, 36--40.

\bibitem{shapiro2}
 Shapiro, J.,
 {\it Finite algebras with abelian properties.}
 Algebra Universalis {\bf 25}, (1988), no. 3, 334--364. 


\bibitem{shelah1}
  Shelah, Saharon,
  {\it On a problem of Kurosh, J\'{o}nsson groups, and applications.}
  Word problems, II (Conf. on Decision Problems in Algebra, Oxford, 1976),
  pp. 373--394, Stud. Logic Foundations Math., {\bf 95},
  North-Holland, Amsterdam-New York, 1980.
  


  



\bibitem{taylor}
  Taylor, Walter,
  {\it Residually small varieties.}
  Algebra Universalis {\bf 2} (1972), 33--53.
  

\end{thebibliography}

\end{document}